\pdfoutput=1
\documentclass[a4paper,12pt]{amsart}
\usepackage{amssymb,amsthm}
\usepackage{mathtools}          
\usepackage{thmtools}
\usepackage{letltxmacro}        
\LetLtxMacro{\oldequiv}{\equiv}
\usepackage{xspace}
\usepackage{xifthen}
\usepackage{pgfplots,tikz}
\usepackage{tikz-cd}
\pgfplotsset{compat=newest}
\usepackage{enumitem}
\usepackage{mathrsfs}

\usetikzlibrary{decorations.markings,decorations.pathreplacing,matrix,arrows,chains,positioning,scopes}

\usepackage{hyphenat}
\usepackage[backref=true,
            bibstyle=numeric-dashed,
            citestyle=numeric,
            maxnames=10,
            dashed=true,
            backend=biber,
            bibencoding=utf8]{biblatex}
\usepackage[colorlinks,
            citecolor=linkcolor,
            linkcolor=linkcolor,
            urlcolor=linkcolor,
            unicode,
            pdfauthor={CAS},
            pdftitle={Construction of the Circle in UniMath},
            pdfsubject={Mathematics},
            pdfkeywords={type theory,univalence foundations,circle}]{hyperref}
\usepackage[capitalize]{cleveref}
\definecolor{linkcolor}{rgb}{0,0,0.5}

\newcommand{\DELETE}[1]{}

\declaretheoremstyle[headfont=\normalfont\bfseries,bodyfont=\itshape]{cas-thm}
\declaretheoremstyle[headfont=\normalfont\bfseries]{cas-def}
\declaretheorem[parent=section,style=cas-thm]{theorem}
\declaretheorem[sibling=theorem,style=cas-thm]{lemma}

\declaretheorem[sibling=theorem,style=cas-def]{definition}

\numberwithin{equation}{section}

\makeatletter
\def\condlabel#1#2{\begingroup
  #2%
  \protected@edef\cref@currentlabel{[cond][][]#2}%
  \phantomsection\label[cond]{#1}%
  \endgroup
}
\makeatother
\crefformat{cond}{(#2#1#3)}

\newcommand{\invo}{^{-\mathrm{o}}}
\newcommand{\refl}[1]{\mathop{{\mathrm{refl}}_{#1}}}
\newcommand{\reflo}[1]{\mathop{{\mathrm{refl}}_{#1}}} 
\newcommand{\refloi}[1]{\mathop{\mathrm{refl}\invo_{#1}}}

\newcommand{\trp}[1]{\mathop{\mathrm{trp}_{#1}}}

\newcommand{\zpos}{\mathop{\mathrm{pos}}}
\newcommand{\zneg}{\mathop{\mathrm{neg}}}
\newcommand{\zzero}{\mathop{\mathrm{zero}}}

\newcommand{\ap}[1]{\mathop{{\mathrm{ap}}_{#1}}}
\newcommand{\apd}[1]{\mathop{{\mathrm{apd}}_{#1}}}

\newcommand{\apc}{\mathop{\mathrm{ap}_{\ct}}}

\newcommand{\id}{\mathord{\mathrm{id}}}
\newcommand{\pt}{{\mathord{\mathrm{pt}}}}

\newcommand{\Prop}{\mathord{\mathrm{Prop}}}

\newcommand*{\texthyphen}{\textup{\kern.5pt-\kern-.5pt}}
\newcommand{\CircleAlg}{\mathord{\mathbb{S}\texthyphen\mathrm{Alg}}}
\newcommand{\CircleFibAlg}{\mathord{\mathbb{S}\texthyphen\mathrm{Fib}\texthyphen\mathrm{Alg}}}

\newcommand{\ZZ}{\mathbb{Z}}
\newcommand{\NN}{\mathbb{N}}

\newcommand{\ct}{*}
\newcommand{\cto}{\mathbin{*_{\mathrm{o}}}}
\newcommand{\cp}[1]{\mathrm{cp}_{#1}}
\newcommand*{\dblslash}{\mathbin{/\kern-3pt/}}
\newcommand*{\Shom}{\mathop{\mathbb{S}\texthyphen\mathrm{hom}}}
\newcommand*{\Ssect}{\mathop{\mathbb{S}\texthyphen\mathrm{sect}}}
\newcommand{\cE}{\mathcal{E}}
\newcommand{\sE}{\mathscr{E}}

\renewcommand{\equiv}{\simeq}

\newcommand{\jdeq}{\oldequiv}
\newcommand{\defeq}{\vcentcolon\jdeq}

\DeclarePairedDelimiter\Trunc{\lVert}{\rVert}
\DeclarePairedDelimiter\trunc{\lvert}{\rvert} 

\DeclarePairedDelimiterX\setof[2]\lbrace\rbrace{#1 \mid #2}

\newcommand*{\UniMath}{\emph{UniMath}}

\makeatletter
\def\mathrule#1#2{%
  \@tempdimb=\dimexpr#2
  \@tempdimc=0.5\@tempdimb
  \@tempdima=\dimexpr#1-\@tempdimc
  \hbox to #1{%
    \pdfliteral{
      q []0 d
      1 J 
      \strip@pt\@tempdimb\space w \strip@pt\@tempdimc\space 0 m
      \strip@pt\@tempdima\space 0 l S Q }}}
\makeatletter
\newcommand*{\blank}{\mathord{\mathchoice%
    {\mathrule{.5em}{.5pt}}
    {\mathrule{.5em}{.5pt}}
    {\mathrule{.4em}{.4pt}}
    {\mathrule{.3em}{.3pt}}
  }}

\newcommand{\ev}{\mathrm{ev}}

\newcommand{\iscontr}{\mathrm{isContr}}

\newcommand{\isEq}{\mathrm{isEquiv}}

\DeclareMathOperator\pr{pr}

\newcommand{\TorZ}{{\mathrm{T}\ZZ}} 

\newcommand{\UU}{\mathcal{U}}
\newcommand{\UUp}{\UU_*}

\newcommand{\Sc}{{\mathbb{S}^1}}

\newcommand{\Zloop}{{\circlearrowleft}} 
\newcommand{\Sloop}{{\circlearrowleft}}

\newcommand{\ptoe}[1]{\tilde {#1}}

\newcommand{\zet}{\mathbb{Z}}

\newcommand{\ind}{\mathop{\mathrm{ind}}}

\newcommand{\shorthash}{e47ce20}

\newcommand{\coqdocbasebaseurl}{https://unimath.github.io/doc/UniMath/\shorthash/}

\newcommand{\coqdocbaseurl}{\coqdocbasebaseurl UniMath.} 
\newcommand{\urlhash}{\#}

\newcommand{\coqdocurl}[2]{\coqdocbaseurl #1.html\urlhash #2}

\newcommand{\nolinkcoqident}[1]{\nolinkurl{#1}} 
\makeatletter
\newcommand{\coqident}{\begingroup\@makeother\#\@coqident}
\newcommand{\@coqident}[3][]{
  \ifthenelse{\isempty{#2}}%
  {\nolinkcoqident{#3}}
  {\ifthenelse{\isempty{#1}}%
  {\href{\coqdocurl{#2}{#3}}{\nolinkcoqident{#3}}}
  {\href{\coqdocurl{#2}{#3}}{\nolinkcoqident{#1}}}}
\endgroup}
\newcommand{\coqfile}[2]{%
  \ifthenelse{\isempty{#1}}%
  {\href{\coqdocbaseurl #2.html}{#2.v}}%
  {\href{\coqdocbaseurl #1.#2.html}{#2.v}}} 
\makeatother

\hypersetup{pdfauthor={Bezem-Buchholtz-Grayson-Shulman},
  pdftitle={Construction of the Circle in UniMath}}
\begin{document}

\title{Construction of the circle in \UniMath}
{
    \author{Marc Bezem}
    \address{Department of Informatics, University of Bergen}
    \email{\href{mailto:marc.bezem@uib.no}{marc.bezem@uib.no}}
    \urladdr{\url{https://www.ae-info.org/ae/Member/Bezem\_Marc}}
}
{
    \author{Ulrik Buchholtz}
    \address{Department of Mathematics, Technical University of Darmstadt}
    \email{\href{mailto:ulrikbuchholtz@gmail.com}{ulrikbuchholtz@gmail.com}}
    \urladdr{\url{https://www2.mathematik.tu-darmstadt.de/~buchholtz/}}
}
{
    \author{Daniel R. Grayson}
    \address{Department of Mathematics, University of Illinois at Urbana-Champaign}
    \curraddr{2409 S.\ Vine St., Urbana, Illinois 61801, USA}
    \email{\href{mailto:danielrichardgrayson@gmail.com}%
      {danielrichardgrayson@gmail.com}}
    \urladdr{\url{http://dangrayson.com/}}
}
{
    \author{Michael Shulman}
    \address{Department of Mathematics, University of San Diego}
    \email{\href{mailto:shulman@sandiego.edu}{shulman@sandiego.edu}}
    \urladdr{\url{http://www.sandiego.edu/~shulman/}}
}

\date{\today}

\begin{abstract}
We show that the type $\TorZ$ of $\ZZ$-torsors has the dependent universal property of the circle,
which characterizes it up to a unique homotopy equivalence.
The construction uses Voevodsky's Univalence Axiom and propositional truncation,
yielding a stand-alone construction of the
circle not using higher inductive types.
\end{abstract}

\maketitle
\markleft{BEZEM BUCHHOLTZ GRAYSON SHULMAN}
\markright{CONSTRUCTION OF THE CIRCLE}
\tableofcontents

\section{Introduction}

In the standard set-theoretic foundations of mathematics, the sets have elements, which are themselves sets.  A set has no additional structure
connecting its various elements to each other.  An equation between two sets is always a proposition, and equality between two elements of a set
is independent of the ambient set.  By contrast, in homotopy type theory and in Voevodsky's univalent foundations, the \emph{types} are the
fundamental objects, serving to classify the objects of mathematics.  They have elements.  Equations are used only to compare two elements of
the same type, and the equations are not always propositions.  Types thus behave much like spaces, viewed from the point of view of homotopy
theory, due to the promotion of isomorphisms between objects of the same type to equalities (true equations), together with the intuition that
an equality between two objects of the same type is like a path between two points in the same space.

\emph{Synthetic homotopy theory} is the study of the homotopy theoretic properties of types.  It is a fruitful one, because it turns out that many
of the most basic results of standard homotopy theory have true analogues for types.  This is true even though the framework is based purely on
logical principles, rather than implemented like traditional homotopy theory in terms of topological spaces and continuous maps or combinatorial
structures such as simplicial sets and fibrant replacement.

In standard homotopy theory it is well known that the \emph{classifying space} of the group $\ZZ$ of integers is
homotopy equivalent to the circle.  We aim to reproduce that result in synthetic homotopy theory in a way that doesn't presuppose the existence of a circle.

The traditional algebraic notion of $\ZZ$-torsor yields a category whose objects are the $\ZZ$-torsors.
Moreover, this category is a groupoid (all arrows are isomorphisms), and
one specific object---namely $\ZZ$ itself---is the trivial $\ZZ$-torsor,
whose automorphism group is $\ZZ$,
and to which every other object is merely isomorphic.
In univalent type theory, this category is thus
completely captured by the type $\TorZ$ of all $\ZZ$-torsors, which
is a connected pointed $1$-type whose fundamental group is $\ZZ$;
we may call it the \emph{classifying type} of $\ZZ$.
In this paper we show that $\TorZ$ behaves the way a circle ought to behave,
by establishing that maps from it to other types (or families of types) are freely determined
by the destinations of the base point and the canonical loop at the base point (corresponding to the element $1$ of $\ZZ$).
The proof is constructive, in that it does not appeal to the axiom of choice or the law of the excluded middle.

There are various types equivalent to the type of $\ZZ$-torsors, and thus they also provide constructions of circles: all one needs is a
connected pointed type whose automorphism groups are isomorphic to $\ZZ$.  For example, if one arising from geometry is desired, one may
consider the type consisting of all frieze patterns in a Euclidean plane\footnote{%
  We say \emph{a} Euclidean plane instead of \emph{the} Euclidean
  plane to indicate that the type is actually the type of pairs consisting of a Euclidean plane and such a frieze pattern in it.}
formed from a
linear collection of evenly spaced copies of the letter $F$.
\[
  \cdots FFFFFFFFFFFFFFFFFFFFFFFFFFFFFFFF \cdots
\]

We have formalized the result in \UniMath{}, a name which refers both to Voevodsky's (univalent) foundation of mathematics based on a formal
type theoretic language and to a particular repository \cite{UniMath} of formalized proofs, initiated by him, encoded in the language of the
proof assistant \emph{Coq}.
See \cite{VV-UniMath-1} for an overview by Voevodsky of it.

The other standard way to construct a circle in type theory uses \emph{higher inductive types}, where one posits a new type $\Sc$, an element
$\pt : \Sc$, a path $\Sloop : \pt = \pt$, an induction principle (for defining functions from it), and nothing more.
Adding higher inductive types to the system would give another construction of the circle, equivalent to the one described above.
Voevodsky chose not to include higher inductive types, nor inductive types,\footnote{\UniMath{} accepts just a few specific types
  and type formers that can be introduced as inductive types, namely: the finite types of cardinality at most 2, the natural numbers, binary
  coproducts, sums of families of types, and identity types.},
in \UniMath{} for two chief reasons.
Firstly, it was (and still is) not clear how to specify precisely what a definition of a higher inductive type consists of: see
\cite{HoTT-Agda-HIT,1705.07088} for two approaches.
Secondly, Voevodsky planned to prove consistency of \UniMath{} in a series of papers\footnote{Early work \cite{1211.2851} with
  Kapulkin and Lumsdaine to establish consistency left him unsatisfied, so he embarked on his own series of papers, to be realized with more details
  exposed.} and realized that adding inductive types and higher inductive types to the system would dramatically increase the length of the
series and the burden of writing them; as it was, he died before completing the project.
At his death the series was well underway but far from complete: \cite{103,109,VV-relmonad,112,39,MR3607209,MR3607210}.
His approach in the series is to interpret each of the basic constructions of the formal system as a traditional mathematical construction: a
context is interpreted as a fibrant simplicial set, a type in a context is interpreted as a fibration of fibrant simplicial sets, an element of
a type in a context is interpreted as a section of such a fibration, and so on.

More recent work \cite{initiality-formalization,initiality}, unpublished but formalized, has established the initiality principle for the type
theory used in \UniMath, which Voevodsky had emphasized as a crucial unproven step.  That, together with the simplicial set model described in
\cite{1211.2851}, may arguably be regarded as establishing the consistency result that Voevodsky sought.

We turn now to the content of this paper.

As a prerequisite we require, of the reader, a working knowledge of homotopy type theory
as described in, for example, the first four chapters of \cite{hottbook}.
In the next \cref{sec:statement}, more precisely in \cref{eq:circle-induction},
we give a precise formulation of the main result.

Preparing for the proof of the main result,
we give in \cref{sec:auxiliaries} some auxiliary results
that are not in the first four chapters of \cite{hottbook}.
The full proof of the main result can be found in \cref{sec:ZTorsors},
more precisely in \cref{sec:TorZ-induction}.
In \cref{sec:topos} we discuss the interpretation of the main result
in higher toposes, before we conclude in \cref{sec:conclusion}.

\section{Precise formulation and discussion}
\label{sec:statement}

\subsection{Preview of the type of \texorpdfstring{$\ZZ$}{Z}-torsors}
\label{sec:preview-Z-tors}

Recall that for a group $G$, a \emph{$G$-torsor} consists of a nonempty set $X$ and
a left action of $G$ on it that is free and transitive.
This means that for all $x,y \in X$ there is a unique $g \in G$ such
that $g \cdot x = y$.
Any element $x_0 \in X$ yields a bijection $e: G \to X$ given by $g \mapsto g \cdot x_0$.
If the operation in the group $G$ is written additively, we'll echo that by writing the torsor operation additively: $g + x$.

In the case where $G$ is $\ZZ$, then since $\ZZ$ is a free group with one generator, an action of $\ZZ$ on a set $X$ is completely
determined by the action of the integer $1$---it is a bijection $X \xrightarrow{\cong} X$ which we denote by $f$.
Fixing $x_0\in X$, the corresponding bijection $e : \ZZ \to X$ then satisfies
  $e(s(n)) = (1+n) + x_0 = 1 + (n + x_0) = f(e(n))$ for all
integers $n$.
So the following diagram commutes.
\[
\begin{tikzpicture} 
   \matrix (m)
   [matrix of math nodes, row sep=2em, column sep=3em, ampersand replacement=\&]
    {\zet \& \zet \\ X \& X \\ };
\draw[->] (m-1-1) -- (m-1-2) node[midway,below] {$s$} node[midway,above] {$\cong$};
\draw[->] (m-1-1) -- (m-2-1) node[midway,left] {$e$} node[midway,right] {$\cong$};
\draw[->] (m-1-2) -- (m-2-2) node[midway,left] {$e$} node[midway,right] {$\cong$};
\draw[->] (m-2-1) -- (m-2-2) node[midway,below] {$f$} ;
\end{tikzpicture}
\]
Conversely, we get a unique $\ZZ$-torsor from any pair $(X,f)$, where $X$ is a nonempty set and $f : X \to X$ is a function, such that there merely
exists a bijection $e$ that makes the diagram above commute, by transporting the structure of $\zet$ as a trivial $\ZZ$-torsor along $e$, thereby
expressing $e$ as an isomorphism of $\ZZ$-torsors.  Independence of the structure on $X$ from the choice of $e$ is established by observing that
$n + x_0 = f^n(x_0)$ for all $n \in \ZZ$.
Alternatively, one could use the existence of $e$ to show that $f$ is a bijection and $X$ is nonempty,
define an action of $\ZZ$ on $X$ by setting $n + x := f^n(x)$ for any $n \in \ZZ$, and then use the existence of $e$ again to show that the
action is free and transitive.

With the above considerations in mind, and recalling that, in the presence of univalence, isomorphisms correspond to identities,
we consider pairs $(X,f)$ of type $\sum_{X:\UU}(X\to X)$
and introduce \emph{the pointed type of $\ZZ$-torsors} by adopting the following definitions
(cf.~\cref{def:TorZ}).
\begin{align*}
  \TorZ &\defeq \sum_{(X,f)}\Trunc{(\zet,s)=(X,f)}  \\
  \pt & \defeq ((\zet,s),\trunc{\refl{(\zet,s)}}) : \TorZ
\end{align*}

{We claim}
that the type $\pt =_\TorZ \pt$ is equivalent to $\zet$.
To see that,
begin by observing that every function $f:\zet\to\zet$ commuting with $s$
is propositionally equal to $s^{f(0)}$.  The first projection
acting on $p : \pt =_\TorZ \pt$ gives a path $\pr_1(p) : \zet=\zet$
such that the corresponding transport function $\pr_1(p)_* : \zet\to\zet$
is a bijection commuting with $s$. We evaluate this bijection {at} $0$.
Let $\ev_0(p) \defeq \pr_1(p)_*(0)$ for all $p : \pt =_\TorZ \pt$,
then $\ev_0 : (\pt =_\TorZ \pt) \to \zet$.  Combining these observations {and}
using the univalence axiom, one proves that $\ev_0$ is an equivalence.

Since $\zet$ is a set, it follows that the underlying set of the fundamental group of $(\TorZ,\pt$) is equivalent to $\zet$,
and with a bit more work, that the fundamental group of $(\TorZ,\pt)$ is isomorphic to $\ZZ$.
However, we do not need this fact here.

We define $1 \defeq s(0)$.
The preimage of $1$ under the equivalence $\ev_0$ is a
natural generating path of $\pt =_\TorZ \pt$, so we define
${\Zloop}\defeq \ev_0^{-1}(1)$, which we call the \emph{loop}
of $\TorZ$. Alternatively, we could have obtained $\Zloop$
directly from $s: \zet\equiv\zet$ by applying the univalence axiom
(with some easy add-ons, e.g., proving that $s$ commutes with itself).

\subsection{The type of circles}
\label{sec:circles}

In order to explain our main result,
that $\TorZ$ behaves like the circle understood as a higher inductive type
with constructors $\pt : \TorZ$ and $\Zloop : pt = \pt$,
we introduce, following \cite{sojakova:hits-hias}, the type of \emph{circle algebras},
consisting of a type together with a point and a loop at that point.
\[
  \CircleAlg \defeq \sum_{C:\UU}~\sum_{c:C}~ c = c.
\]
The type of \emph{fibered circle algebras} over a circle algebra
$\mathcal C \jdeq (C,c,s)$ is defined to be the dependent version thereof.
\[
  \CircleFibAlg(\mathcal C) \defeq
  \sum_{A : C \to \UU}~\sum_{a:A(c)}~ a=^A_s a.
\]
As explained and proved in \cite[Thm.~50]{sojakova:hits-hias},
given a circle algebra $\mathcal C \jdeq (C,c,s)$,
there are various equivalent ways to state that it behaves like the circle,
including\footnote{%
  The dependent universal property is not discussed in~%
  \cite{sojakova:hits-hias}, however, it is easily seen
  to be equivalent to the induction principle.}:
\begin{description}
\item[Homotopy initiality]
  The type of algebra homomorphisms
  from $\mathcal C$ to any other circle algebra $\mathcal A\jdeq(A,a,p)$
  is contractible, where this type of algebra homomorphisms is
  \begin{equation}\label{eq:circle-hom}
    \Shom(\mathcal C,\mathcal A) \defeq
    \sum_{f : C \to A}~\sum_{r:f(c)=a}~ \ap{f}(s)=^{\tilde A}_r p,
  \end{equation}
  where $\tilde A(y) \defeq (y=y)$ for $y:A$.
  We recall the basics of paths over a path, and the accompanying
  operations, in \cref{sec:pathovers} below.
  (The type $\Shom(\mathcal C,\mathcal A)$ is equivalent to the type
  \[
    \sum_{f : C \to A} (f(c),\ap{f}(s)) = (a,p),
  \]
  where the pairs are regarded as elements of $\sum_{y:A} \tilde A(y)$.)
\item[Universal property]
  For any type $A:\UU$, the map that evaluates a function $f : C \to A$
  at $c$ and $s$,
  \begin{equation}\label{eq:eval}
    (C \to A) \to \sum_{a : A} a=a,
    \qquad
    f \mapsto (f(c), \ap{f}(s)),
  \end{equation}
  is an equivalence.
\item[Induction principle]
  Any fibered circle algebra $\mathcal A\jdeq (A,a,p)$
  over $\mathcal C$ has an algebra section,
  where the type of algebra sections is
  \begin{equation}\label{eq:circle-sect}
    \Ssect(\mathcal A) \defeq
    \sum_{f: \prod_{z:C} A(z)}~
    \sum_{r: f(c) = a}~
    \apd{f}(s) =^{\tilde A}_r p,
  \end{equation}
  where $\tilde A(y) \defeq (y =^A_s y)$ for $y:A(c)$.
  (Here $\apd{}$ is as defined in \cite[2.3]{hottbook};
  the type $\Ssect(\mathcal A)$ is equivalent to the type
  \[
    \sum_{f: \prod_{z:C} A(z)} (f(c),\apd{f}(s)) = (a,p),
  \]
  where the pairs are regarded as elements of $\sum_{y:A(c)}\tilde A(y)$.)
\item[Dependent universal property]
  For any type family $A : C \to \UU$, the map that evaluates a function
  $f : \prod_{z:C} A(z)$ at $c$ and $s$,
  \begin{equation}\label{eq:dep-eval}
    \biggl(\prod_{z:C} A(z)\biggr) \to \sum_{a:A(c)} a=^A_s a,
    \qquad
    f \mapsto (f(c), \apd{f}(s)),
  \end{equation}
  has a section.
\end{description}
From the equivalence of the types encoding the four principles above,
it also follows that the type of induction terms
for $\mathcal C$,
\begin{equation}
  \label{eq:circle-induction-gen}
  \prod_{A: C\to\UU}~
  \prod_{a: A(c)}~
  \prod_{p: a=^A_s a}~
  \Ssect(A,a,p),
\end{equation}
which is equivalent to the type of sections of \eqref{eq:dep-eval}
for all families $A$, is a proposition.

The type of circles, then, is the subtype of $\CircleAlg$
corresponding to any of these definitions.
As in~\cite[Sec.~9.8]{hottbook},
we can prove a structure identity principle for circle algebras and circles,
viz., for circle algebras $\mathcal C\jdeq(C,c,s)$ and
$\mathcal C'\jdeq(C',c',s')$, the canonical function
\[
  (\mathcal C =_{\CircleAlg} \mathcal C')
  \to
  \sum_{f : C \to C'}~\sum_{r:f(c)=c'}~(\ap{f}(s)=_r^{\tilde C'} s')
  \times \isEq(f),
\]
where $\tilde C'(y) \defeq (y=y)$ for $y:C'$,
is an equivalence.
By homotopy initiality for circles,
it then follows straight-forwardly that the type of circles
is a proposition, i.e.,
independently of whether there are any circles,
any two circles are uniquely identifiable.

\subsection{\texorpdfstring{$\TorZ$}{TZ} is a circle}
\label{sec:TorZ-circle}

With this in hand, we can outline our route to proving
that the circle algebra $(\TorZ,\pt,\Zloop)$ is a circle.

As a warm-up, we first prove in \cref{sec:TorZ-recursion}
the \emph{recursion principle},
which states that, given a type $A$, an element $a$ of $A$ and a
path $p:a=_A a$, one can construct a function $f:\TorZ\to A$
with a path $r:f(\pt)=a$ such that $\ap{f}(\Zloop)=^{\tilde A}_r p$,
or, equivalently, $\ap{f}(\Zloop) = r \ct (p \ct r^{-1})$.
This construction was formalised by Grayson
in 2014, see \cite[\href{https://github.com/UniMath/UniMath/blob/master/UniMath/SyntheticHomotopyTheory/Circle.v}{Circle.v}]{UniMath}.\footnote{Our construction here uses the
same basic idea but manages the computations involved in establishing the
induction principle differently.}
This however only proves \emph{weak} initiality, or equivalently,
only gives a \emph{section} of the evaluation map in \eqref{eq:eval}.

In \cref{sec:TorZ-induction} we then prove the induction principle,
in which $A$ is not a type but a type family over $\TorZ$.
On the basis of Grayson's construction of the recursion principle,
Shulman sketched an approach to the induction principle, see \cite{circleind-Mike}.
Independently of this, but also departing from Grayson's proof of the
recursion principle, Buchholtz and Bezem found the construction
presented in this paper, which has subsequently been formalized
by Grayson \cite[\href{https://github.com/UniMath/UniMath/blob/master/UniMath/SyntheticHomotopyTheory/Circle2.v}{Circle2.v, Theorem circle\_induction}]{UniMath} in \UniMath.

Spelling out the induction principle, we have to construct a
function of the following type:
\begin{equation}
  \label{eq:circle-induction}
  \thinmuskip=10mu              
  \prod_{A: \TorZ\to\UU}
  \prod_{a: A(\pt)}
  \prod_{p: a=^A_{\Zloop} a}
  \sum_{f: \prod_{z:\TorZ} A(z)}
  \sum_{r: f(\pt)= a}
  \apd{f}(\Zloop) =^{\tilde A}_r p
\end{equation}
(Recall that $\tilde A(y) \jdeq (y =^A_\Zloop y)$ for $y : A(\pt)$.)

\begin{figure}
  \centering
  \begin{tikzpicture}
    \foreach \x in {0,2,4}
    { \begin{scope}[shift={(\x,0)}]
        \foreach \y in {0,1,3}
        { \begin{scope}[shift={(0,\y)}]
            \draw (0,0) .. controls ++(170:-.3) and ++(210: .4) .. (1,0)
                        .. controls ++(210:-.4) and ++(170: .3) .. (2,0);
          \end{scope} }
        \node[fill,circle,inner sep=1pt,label=below:$\pt$] at (1,0) {};
        \node[fill,circle,inner sep=1pt,label=below left:$a$] at (1,1.6) {};
        \node[fill,circle,inner sep=1pt] at (1,2.5) {};
        \draw (1,1) -- (1,3);
        \node at (1,3.5) {$A(\pt)$};
        \draw[dashed] (0,2.4) .. controls ++(150:-.5) and ++(190: .3) .. (1,2.5)
                              .. controls ++(190:-.3) and ++(150: .5) .. (2,2.4);
        \draw[thick,dashed] (1,2.5) -- (1,1.6);
        \node at (1.2,2) {$r$};
      \end{scope} }
    \foreach \x in {0,2}
    { \begin{scope}[shift={(\x,0)}]
        \draw (1,1.6) .. controls ++(140:-.4) and ++(190:.3) .. (2,1.8)
                      .. controls ++(190:-.3) and ++(140:.4) .. (3,1.6);
        \node at (2,1.5) {$p$};
        \node at (2,-.3) {$\Zloop$};
        \node at (2.2,2.65) {$\scriptstyle \apd{f}(\Zloop)$};
      \end{scope} }
    \foreach \y in {0,2} {
      \node at (-.5,\y) {$\cdots$};
      \node at (6.5,\y) {$\cdots$};
    }
    \node (T) at (-1.5,0) {$\TorZ$};
    \node (A) at (-1.5,2) {$A$};
    \node (f) at (0.3,2.6) {$f$};
  \end{tikzpicture}
  \caption{A visualization of circle induction principle}
  \label{fig:circle-induction}
\end{figure}
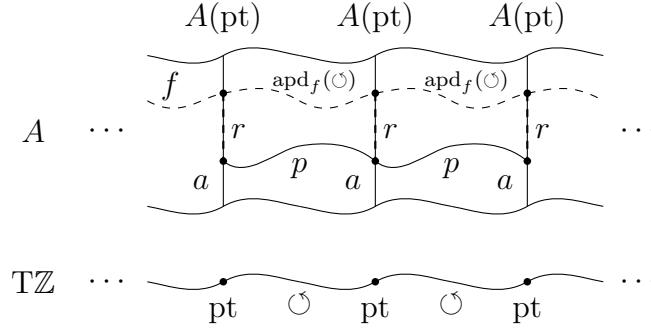

The type in \cref{eq:circle-induction} is illustrated in \cref{fig:circle-induction}.
We think of $\TorZ$ as a circle (and prove that it is),
but we illustrate it (and objects depending on it) as periodically recurring,
in order to make clearer diagrams.
We draw the type family $A$ as a periodic fibration over this circle,
and the goal is then to produce from the point $a$ and the path $p$ over $\Sloop$,
the section $f$, the path $r$, and a final element inhabiting
the type $\apd{f}(\Zloop) =^{\tilde A}_r p$,
which corresponds to filling the inside of the curvilinear quadrilateral in \cref{fig:circle-induction}.
(This will be made precise in course of the argument below,
in terms of compositions of paths over paths.)

Our construction uses the propositional truncation operation $\Trunc{\blank}$,
as in \cite[6.9]{hottbook},
but we do not require that the dependent eliminator
witnessing the induction principle of $\Trunc{\blank}$
computes judgmentally on the point constructor $\trunc{\blank}: X \to \Trunc X$.
In the \UniMath{} formalization, propositional truncation is constructed as
(the propositional resizing of)
\begin{equation}\label{eq:trunc-impred}
  \Trunc X \defeq \prod_{P:\Prop}((X \to P) \to P),
\end{equation}
where $\Prop$ is the type of all propositions (in a chosen universe).
Using this construction, the corresponding non-dependent eliminator
will compute judgmentally on the point constructor,
but the dependent eliminator will not.
As we explain below, this has the consequence that
in \UniMath{}, the non-dependent circle eliminator that we construct for $\TorZ$
will then compute judgmentally on $\pt$,
but the dependent one will not.

If the underlying type theory has propositional truncation with a dependent eliminator
that computes judgmentally on the point constructor, as in \cite[6.9]{hottbook},
then our construction for the induction principle in \cref{eq:circle-induction} simplifies.
Its application to $A,a,p$ of appropriate types
is a triple $(f,r,q)$ where $f$ is a section of $A$ and $r : f(\pt)= a$.
In this case, $f(\pt)\jdeq a$, but in general we won't have $r \jdeq \refl{a}$.
However, the type $\apd{f}(\Zloop) =^{\tilde A}_{\bar r} p$ makes sense
for all $\bar r : a=a$, and we'll have that $r = \refl{a}$,
so we may transport $q$ to obtain an element of type
$\apd{f}(\Zloop) =^{\tilde A}_{\refl{a}} p$.
This type is judgmentally equal to $\apd{f}(\Zloop)=p$.

In a type theory where higher inductive types are admitted,
one can introduce a circle $\Sc$ as a higher inductive type,
as in \cite[6.1]{hottbook},
and this will easily be shown to be equivalent to $\TorZ$,
without relying on the results of our paper.
The equivalence can be established by observing that $\TorZ$ and $\Sc$
are pointed connected types with loop spaces equivalent to $\zet$.
(For the interested reader: apply \cite[Lemma 7.6.2]{hottbook} with $n=-2$
to reduce to action on paths. Then use connectedness to strengthen
this result from embeddings to equivalences.
Finally, again using connectedness, reduce to the loop spaces
of the respective points, and show they are equivalent to $\zet$ and thus to each other.)
The induction principle \cref{eq:circle-induction} for $\TorZ$ then follows directly.

The reader may wonder whether our result can be generalized to higher spheres.
The answer is no.
One can construct models of type theory with univalent universes and propositional resizing
(viz., essentially of \UniMath{})
that do not have higher inductive types, not even suspensions.
We can take any model and restrict the $n$th universe, $n=0,1,\dots$,
to consist of homotopy $n$-types:
the new zeroth universe $\UU'_0$ consists of the sets of the old universe $\UU_0$,
the new universe $\UU'_1$ consists of the groupoids of $\UU_1$ (hence including $\UU'_0$),
etc.
This construction keeps all the propositions, so it preserves propositional resizing.
By our construction, $\UU'_1$ will contain a circle $\TorZ$,
but it cannot contain the $2$-sphere (the suspension of the circle),
as this is not a groupoid.
Because the $2$-sphere in topology is not an $n$-type for any $n$,
the new model will not contain a $2$-sphere.
The restricted model also shows that one cannot construct a circle in
the lowest universe in \UniMath{}.

\section{Auxiliary results}
\label{sec:auxiliaries}

In this section we present some additional results that are needed in the sequel.

\subsection{Identifying elements in members of families of types}
\label{sec:pathovers}

All proofs in this subsection are by (nested) induction as indicated in the text,
and can moreover be found in \cite[\href{https://github.com/UniMath/UniMath/blob/master/UniMath/MoreFoundations/PathsOver.v}{PathsOver.v}]{UniMath}.

Let $A : \UU$, $B : A \to \UU$, $a_i:A$, $b_i:B(a_i)$ for $i=1,2$, and $p : a_1 = a_2$.
We are interested in identifications of $b_1$ and $b_2$ relative to this data.
We cannot in general form the type $b_1 = b_2$ as their types may be different.
There are several ways to solve this problem. One of them is to transport $b_1$ along
$p$ and form an identity type in $B(a_2)$. Another way would be to consider
identifications $(a_1,b_1) = (a_2,b_2)$ in $\sum_{x:A} B(x)$ and require that the
action of the first projection on such identifications is equal to $p$.
These two ways are equivalent.
The former way is easier to work with and will be the one we choose here.

\begin{definition}\label{def:pathover}
Let $A : \UU$, $B : A \to \UU$, $a_i:A$, $b_i:B(a_i)$ for $i=1,2$, and $p : a_1 = a_2$.
Define the transport function $\trp{B,p} : B(a_1)\to B(a_2)$ by induction on $p$,
setting $\trp{B,\refl{a_1}}(b_1)\defeq b_1$. This is indeed well-typed since
$B(a_1)\jdeq B(a_2)$ in this case.
Now define the type $b_1 =^B_p b_2$ as $\trp{B,p}(b_1) = b_2$.
An element of $b_1 =^B_p b_2$ is called
a \emph{path} from $b_1$ to $b_2$ \emph{over} $p$.
Note that $(b_1 =^B_{\refl{a_1}} b_2) \jdeq (b_1 =_{B(a_1)}  b_2)$.
\end{definition}

Many of the operations on paths have their counterpart for paths over paths.
We define the unit path over a path, composition of paths over paths, and reversal of paths over paths.

\begin{definition}\label{def:pathoveralgebra}
  Let $A : \UU$, $B : A \to \UU$, $a_i:A$, $b_i:B(a_i)$ for $i=1,2,3$, and
  $p_i : a_i = a_{i+1}$ for $i=1,2$. We define:
  \begin{enumerate}[topsep=3pt]
  \item \emph{Unit} $\reflo{b_1} : b_1 =^B_{\refl{a_1}} b_1$;
  \item \emph{Composition of paths over paths}
    $\cto : b_1 =^B_{p_1} b_2 \to b_2 =^B_{p_2} b_3 \to b_1 =^B_{p_1 \ct p_2} b_3$,
    defined by induction first on $p_2$ and then on $r: b_2 = b_3$, by
    setting $q \cto \reflo{b_2} \defeq q$ for all $q: b_1 =^B_{p_1} b_2$;
  \item \emph{Reversal of paths over paths}
    $(\blank)\invo : b_1 =^B_{p_1} b_2 \to b_2 =^B_{(p_1)^{-1}} b_1$,
    defined by induction first on $p_1$ and then on $r: b_1 = b_2$, by
    setting $\refloi{b_1} \defeq \refl{b_1}$.
  \end{enumerate}
\end{definition}

These operations on paths over paths satisfy many of the laws satisfied by the corresponding operations on paths, after some modification.  We
illustrate the modification required to treat composition.  Suppose we have elements $a_i : A$ for $1 \le i \le 4$, paths $p_i : a_i = a_{i+1}$
for $1 \le i \le 3$, elements $b_i:B(a_i)$ for $1 \le i \le 4$, and paths $q_i : b_i =^B_{p_i} b_{i+1}$ over $p_i$ for $1 \le i \le 3$.
Then we the following two paths over paths.
\begin{align*}
   q_1 \cto (q_2 \cto q_3) & : b_1 =^B_{p_1\ct (p_2\ct p_3)} b_4 \\
  (q_1 \cto q_2) \cto q_3  & : b_1 =^B_{(p_1\ct p_2) \ct p_3} b_4
\end{align*}
Since they are of different types, they cannot be compared directly, but there is an
equivalence $\varepsilon$ of type $\left( b_1 =^B_{p_1\ct (p_2\ct p_3)} b_4 \right) \equiv \left( b_1 =^B_{(p_1\ct p_2) \ct p_3} b_4 \right)$
constructed from the associativity law for paths of type $p_1\ct(p_2\ct p_3) = (p_1\ct p_2)\ct p_3$.
The associativity law for composition of paths over paths is an easily constructed identity of type $\varepsilon(q_1 \cto (q_2 \cto q_3)) = (q_1 \cto q_2)\cto q_3$.
For more information we refer the reader to the repositories with formalized proofs
\cite[\href{https://github.com/UniMath/UniMath/blob/master/UniMath/SyntheticHomotopyTheory/Circle2.v}{Circle2.v}]{UniMath}.

In the rest of this section we work in a context with
$A : \UU$, $B : A \to \UU$, $a_i:A$, $b_i:B(a_i)$ for $i=1,2,3$,
$p_i : a_i = a_{i+1}$ for $i=1,2$,

\begin{lemma}\label{lem:compo-over}
  For every $q : b_1 =^B_{p_1} b_2$, the
  function
  \[
    q \cto ({\blank}) : (b_2 =^B_{p_2} b_3) \to (b_1 =^B_{p_1\ct p_2} b_3)
  \]
  is an equivalence.
\end{lemma}
The proof is by induction on first $p_1$, and then $q$, and finally $p_2$.
Then conclude by reflexivity.

If $p=q$, then we can transport paths over $p$ to paths over $q$.

\begin{definition}\label{def:pathover-change-path}
  For every $p,q:a_1=a_2$ and $2$-dimensional path $\alpha : p = q$,
  transport along $\alpha$
  induces an equivalence $\cp{\alpha}: (b_1 =^B_p b_2) \equiv (b_1 =^B_q b_2)$.
  The function $\cp{\alpha}$ is called \emph{change path}, and is defined
  by induction on $\alpha$, setting $\cp{\refl{p}}\defeq \id_{b_1 =^B_p b_2}$.
\end{definition}

\begin{lemma}\label{lem:functorial-change-path}
  For every $p,q,r:a_1=a_2$ and 2-paths $\alpha : p = q$, $\beta : q = r$,
  we have $\cp{\alpha\ct\beta}=\cp{\beta}\circ\cp{\alpha}$.
\end{lemma}

The proof is by induction on $\beta$ (for right-recursive composition).

\begin{lemma}\label{lem:inv2-change-path}
  For every  $p,q:a_1=a_2$ and 2-path $\alpha : p = q$, taking
  $\alpha^- \defeq \ap{(\blank)^{-1}}(\alpha)$, we have
  $(\cp{\alpha}(\hat p))\invo = \cp{\alpha^-}(\hat p\invo)$
  for every $\hat p: b_1=^B_p b_2$.
\end{lemma}
The proof is by induction on $\alpha$.

\begin{lemma}\label{lem:invlaw-change-path}
  For every  $p :a_1 = a_2$ define $\iota(p): p^{-1}\ct p = \refl{a_2}$
  by induction on $p$, by setting $\iota(\refl{a_1})\defeq \refl{\refl{a_1}}$.
  Then we have $\cp{\iota(p)}(\hat p\invo \cto {\hat p}) = \refl{b_2}$
  for every $\hat p: b_1=^B_p b_2$.
\end{lemma}

\begin{lemma}\label{lem:unitlaw-change-path}
  For every  $p :a_1 = a_2$ define $\gamma(p): \refl{a_1} \ct\, p = p$
  by induction on $p$, setting $\gamma(\refl{a_1})\defeq \refl{\refl{a_1}}$.
  Then we have $\cp{\gamma(p)}(\reflo{b_1} \cto {\hat p}) = \hat p$
  for every $\hat p: b_1=^B_p b_2$.
\end{lemma}

The proof is in both cases by induction on first $p$, and then on $\hat p$.

\begin{definition}\label{lem:compo-ap-ap}
  For every  $p,p':a_1=a_2$, $q,q':a_2=a_3$ and 2-paths
  $\alpha : p = p'$, $\beta : q = q'$, define
  $\apc(\alpha,\beta): (p\ct q)=(p'\ct q')$ by induction
  first on $\beta$ and then on $q$,
  setting $\apc(\alpha,\refl{\refl{a_2}}) \defeq \alpha$.
  (This is well-typed for right-recursive composition.)
\end{definition}

\begin{lemma}\label{lem:compo-change-path}
  For every  $p,p':a_1=a_2$, $q,q':a_2=a_3$ and 2-paths
  $\alpha : p = p'$, $\beta : q = q'$, we have
  $\cp{\apc(\alpha,\beta)}(\hat p \cto \hat q) =
   \cp{\alpha}(\hat p) \cto  \cp{\beta}(\hat q),$
  for every $\hat p: b_1=^B_p b_2$ and $\hat q: b_2=^B_q b_3$.
\end{lemma}
The proof is by induction first on $\beta$, then on $q$, and finally on $\hat q$.

\subsection{Dependent elimination for propositional truncation}
\label{sec:proptrunc}

One extra piece of knowledge we need is the \emph{dependent} elimination
principle for propositional truncation, mentioned in \cite[6.9]{hottbook},
and qualified as `not really useful'.%
\footnote{The dependent elimination principle appears to
  be used later in \cite[Lemma 7.3.3]{hottbook}, though.
}
We use this principle in \cref{lem:dep-elim-TorZ}, and it plays an
essential role in the proof of the induction principle for $\TorZ$.

Recall that we assume the presence of a propositional truncation operation
$\Trunc{\blank}$ (mapping each universe to the propositions therein) equipped with
maps $\trunc{\blank} : X \to \Trunc X$ exhibiting $\Trunc{\blank}$ as a reflection into
propositions, i.e., for any proposition $P$, precomposition with $\trunc{\blank}$
induces an equivalence $(\Trunc{X}\to P) \equiv (X\to P)$.
If the inverse map sends $g:X\to P$ to $\bar g : \Trunc X \to P$
such that $\bar g(\trunc x) \jdeq g(x)$ for all $x:X$,
then we say the propositional truncation operation satisfies the
judgmental computation rule for the non-dependent eliminator~%
(\condlabel{cond:JNE}{JNE}).
This can be achieved either by having $\Trunc{\blank}$ defined as a higher inductive type
or using the impredicative encoding~\eqref{eq:trunc-impred}, as we verify below.

\begin{lemma}\label{lem:proptruncind}
  If $A$ is a type and $B(x)$ is a family of propositions
  indexed by the elements $x$ of $\Trunc A$,
  and we are given $g : \prod_{a:A} B(\trunc a)$,
  then there is a function $f : \prod_{x:\Trunc A} B(x)$.
\end{lemma}
\begin{proof}
  The function $g$ induces a function $h : A \to \sum_{x:\Trunc A}B(x)$
  by setting $h(a) \defeq (\trunc a,g(a))$.
  Since the codomain of $h$ is a proposition,
  we get an induced map $\Trunc A \to \sum_{x:\Trunc A}B(x)$,
  which is automatically a section of $\pr_1 : \sum_{x:\Trunc A}B(x) \to \Trunc A$.
  By the equivalence between sections of first projections
  and dependent functions, we then get the required $f$.
\end{proof}
Note that the resulting function $f$ automatically satisfies
$f(\trunc a) = g(a)$, for each $a:A$,
since each type $B(\trunc a)$ is a proposition.
If we can achieve a judgmental equality here,
then we say that the propositional truncation operation satisfies the
judgmental computation rule for the dependent eliminator~%
(\condlabel{cond:JDE}{JDE}).

Recall that in \UniMath{}, $\Trunc A \defeq \prod_{P:\Prop}((A \to P) \to P)$, and
the function $\trunc{\blank}: A\to\Trunc A$ is
defined by $\trunc a  (P,f) \defeq f(a)$ for all $P: \Prop$, $f: A\to P$, and $a:A$.
Thus we naturally get a recursion principle: if $P: \Prop$, then any
$g: A \to P$ defines an $f : \Trunc A \to P$ by setting $f(x)\defeq x(P,g)$.
Clearly $f$ satisfies $f(\trunc a) \jdeq \trunc a (P,g) \jdeq g(a)$, for all $a:A$.

For induction the situation is different: if $B: \Trunc A \to \Prop$,
then from any $g: \prod_{a:A} B(\trunc a)$ we can construct a function
$f: \prod_{x:\Trunc A} B(x)$, as in the proof of \cref{lem:proptruncind}.
Unwinding the proof, there will be a transport involved in converting
the section to a dependent function, and we won't have $f(\trunc a) \jdeq g(a)$,
as the path of type $\trunc a = \trunc a$ coming from the proof that $\Trunc A$ is
a proposition need not be the path given by reflexivity.

There are several other ways of defining $f$, but none we are aware of
satisfies the computation rule $f(\trunc a)\jdeq g(a)$.  It is currently
unknown whether it is possible to define in \UniMath{} a propositional
truncation with an induction principle that satisfies the computation rule.

\subsection{The integers}
\label{sec:integers}

For convenience we give a direct inductive
definition of the set of integers, and we give
an alternative induction principle in \cref{sec:integers-induction}.
Thus we get a set of integers $\zet$, a constant $0:\zet$,
and a successor function $s:\zet\to\zet$
that is an equivalence.

\begin{definition}\label{def:integers}
Let $\zet$ be the inductive type with the following three constructors:
\begin{enumerate}[topsep=0pt]
\item $\zzero: \zet$ for the integer number zero,
$0 \defeq \zzero$
\item $\zpos: \NN \to \zet$ for positive {integers},
$1 \defeq \zpos(0),\ldots$.
\item $\zneg: \NN \to \zet$ for negative {integers},
$-1 \defeq \zneg(0),\ldots$
\end{enumerate}
\end{definition}

The \emph{embedding} function $i:\NN\to\zet$ is defined by induction,
setting $i(0)\defeq \zzero$, $i(S(n))\defeq \zpos(n)$.
Like the type $\NN$, the type $\zet$ is a set with decidable equality
and ordering relations,
and we denote its elements often in the usual way as $\ldots,-1,0,1,\ldots$.

One well-known equivalence is \emph{negation} ${-}:\zet\to\zet$,
also called \emph{complement}, inductively defined by setting
$-\zzero\defeq \zzero$,
$-\zpos(n)\defeq \zneg(n)$,
$-\zneg(n)\defeq \zpos(n)$.
Negation is its own inverse.

The \emph{successor} function $s:\zet\to\zet$ is defined inductively setting
$s(\zzero)\defeq \zpos(0)$,
$s(\zpos(n))\defeq \zpos(S(n))$,
$s(\zneg(n))\defeq -i(n)$. For example, we have
$s(-1)\jdeq s(\zneg(0))\jdeq -i(0) \jdeq \zzero \jdeq 0$.
{Denoting the successor function on $\NN$ by $S$,
by} induction on $n:\NN$ one proves $s(i(n))=i(S(n))$,
so that one can say that $s$ extends $S$ on the $i$-image of $\NN$.
From now on we will identify $i(n):\zet$ with $n$,
and $-i(n):\zet$ with $-n$, for all $n:\NN$.

The successor function $s$ is an equivalence.
The inverse $s^{-1}$ of $s$ is called the \emph{predecessor} function.
We denote the $n$-fold iteration of $s$ as $s^n$, and
the $n$-fold iteration of $s^{-1}$ as $s^{-n}$.

Addition of integers is defined inductively by setting
$z + \zzero\defeq z$,
$z + \zpos(n)\defeq s^{n+1}(z)$,
$z + \zneg(n)\defeq s^{-(n+1)}(z)$.
From addition and unary $-$ one can define a binary
\emph{substraction} function setting $z-y \defeq z+(-y)$.


Recall the equivalence $\ev_0 : (\pt=_\TorZ \pt)\to\zet$ from
\cref{sec:preview-Z-tors}, which sends $p$ to $\pr_1(p)_*(0)$.
We have $\refl{\pt} : \pt=_\TorZ \pt$, as well as the operations
of \emph{path reversal} and \emph{path composition} as defined
in \cite[2.1]{hottbook}. These satisfy the laws as stated
and proved in \cite[Lemma 2.1.4]{hottbook}, equipping $\pt=_\TorZ \pt$
with a group structure.
The equivalence $\ev_0$ maps $\refl{\pt}$ via $\id_\zet : \zet\to\zet$ to $0$.
As explained above, $\ev_0$ maps any $p : \pt=_\TorZ \pt$
via $s^k : \zet\to\zet$ to some $k:\zet$. Since $(s^k)^{-1} = s^{-k}$
and $s^k \circ s^l = s^{k+l}$, $\ev_0$ transports path reversal
and path composition to negation and addition, respectively.
This means that the entire group structure of $\pt=_\TorZ \pt$
is transported to the usual group structure on $\zet$,
including all the proofs of the group laws in $\zet$.
(The fact that we do not have to reprove the group laws
is one of the benefits of the univalent approach.)

\subsection{Some induction principles for the integers}
\label{sec:integers-induction}

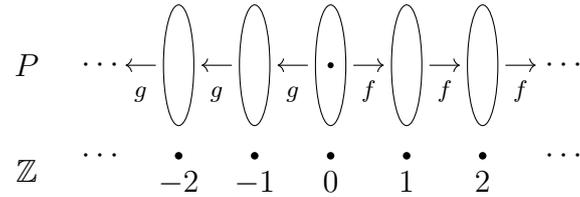
\begin{figure}[b]
  \centering
  \begin{tikzpicture}
    \foreach \x in {-2,-1,0,1,2} {
      \begin{scope}[shift={(\x,0)}]
        \draw (0,0) ellipse (0.2 and 0.8);
        \node[fill,circle,inner sep=1pt,label=below:{$\x$}] (P\x) at (0,-1.2) {};
      \end{scope}
    }
    \foreach \x in {0,1,2} {
      \begin{scope}[shift={(\x,0)}]
        \draw[->] (0.3,0) -- node[below] {\scriptsize$f\strut$} (0.7,00);
      \end{scope}
    }
    \foreach \x in {-3,-2,-1} {
      \begin{scope}[shift={(\x,0)}]
        \draw[->] (0.7,0) -- node[below] {\scriptsize$g\strut$} (0.3,00);
      \end{scope}
    }
    \node[fill,circle,inner sep=.8pt] at (0,0) {};
    \node at (-4,0) {$P$};
    \node at (-4,-1.4) {$\ZZ$};
    \foreach \x in {-3,3.1} {
      \foreach \y in {0,-1.2} {
        \node at (\x,\y) {$\cdots$};
      }
    }
  \end{tikzpicture}
  \caption{Asymmetric integer induction principle}
  \label{fig:integers-induction-asymmetric}
\end{figure}

The definition of $\zet$ yields the following induction principle.
Given $P : \zet \to \UU$, to construct elements $h(z) : P(z)$ for every $z : \zet$,
it suffices to give $h(0): P(0)$ and functions
$f : \prod_{n:\NN}(P(n) \to P(n+1))$ and
$g : \prod_{n:\NN}(P(-n) \to P(-n-1))$,
as illustrated in \cref{fig:integers-induction-asymmetric}.

The function $h:\prod_{z:\ZZ}(P(z))$ defined in this way satisfies
$h(n+1)\jdeq f(n,h(n))$ and $h(-n-1)\jdeq g(n,h(-n))$ for all $n:\NN$.

It is possible to give a more symmetric, but less general induction principle,
if we assume that the functions are equivalences.
In that case we can reorient the $g$'s to point in the same direction as the $f$'s,
allowing us to combine them into a single family $f : \prod_{z:\zet}P(z) \equiv P(z+1)$ of equivalences,
as illustrated in \cref{fig:integers-induction-symmetric}.

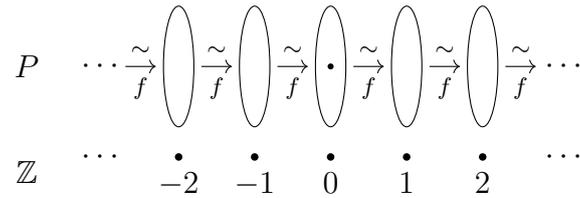
\begin{figure}[b]
  \centering
  \begin{tikzpicture}
    \foreach \x in {-2,-1,0,1,2} {
      \begin{scope}[shift={(\x,0)}]
        \draw (0,0) ellipse (0.2 and 0.8);
        \node[fill,circle,inner sep=1pt,label=below:{$\x$}] (P\x) at (0,-1.2) {};
      \end{scope}
    }
    \foreach \x in {-3,-2,-1,0,1,2} {
      \begin{scope}[shift={(\x,0)}]
        \draw[->] (0.3,0) -- node[above] {\scriptsize$\sim$} node[below] {\scriptsize$f$} (0.7,00);
      \end{scope}
    }
    \node[fill,circle,inner sep=.8pt] at (0,0) {};
    \node at (-4,0) {$P$};
    \node at (-4,-1.4) {$\ZZ$};
    \foreach \x in {-3,3.1} {
      \foreach \y in {0,-1.2} {
        \node at (\x,\y) {$\cdots$};
      }
    }
  \end{tikzpicture}
  \caption{Symmetric integer induction principle}
  \label{fig:integers-induction-symmetric}
\end{figure}

We shall need that in this case, giving an element $h:\prod_{z:\zet}P(z)$
together with identities of type $h(z+1) = f_z(h(z))$ for all $z:\zet$
is equivalent to giving the single element $h(0)$.
We formulate this precisely as follows.

\begin{theorem}\label{thm:integers-univ-symm}
  Let $P : \zet \to \UU$ and $f : \prod_{z:\zet}P(z) \equiv P(z+1)$. The function
  \[
    \varphi : \biggl(\sum_{h:\prod_{z:\zet}P(z)}\prod_{z:\zet}h(z+1) = f_z(h(z))\biggr) \to P(0)
  \]
  that sends $(h,q)$ to $h(0)$ is an equivalence.
\end{theorem}

See \cite[\href{https://github.com/UniMath/UniMath/blob/master/UniMath/SyntheticHomotopyTheory/AffineLine.v}{AffineLine.v, Definition $\mathbb Z$Bi\-Recursion\_weq}]{UniMath}
for the formalization of the proof, and
see \cite[\href{https://github.com/UniMath/UniMath/blob/master/UniMath/SyntheticHomotopyTheory/AffineLine.v}{AffineLine.v, Definition $\mathbb Z$Torsor\-Recursion\_weq}]{UniMath}
for the formalization of a version for arbitrary $\zet$-torsors.

\begin{proof}
  We prove that the fiber over any $p : P(0)$ is contractible.
  We simplify notations a bit by leaving out the types of $h$ and $q$.
  The fiber $\sum_{(h,q)} h(0)=p$ consists of triples $(h,q,r)$ with $r : h(0) = p$.
  By case distinction, $h$ can (equivalently) be split in three parts $(h_-,h_0,h_+)$
  with $h_0 : P(0)$, $h_+ : \prod_{n:\NN}P(n+1)$,  and $h_- : \prod_{n:\NN} P(-n-1)$.
  Since $h(0)=p$ only depends on $h_0$ the pair $(h_0,r)$ with  $r : h_0 = p$
  contracts away, so we're left with the type
  \[
    \varphi^{-1}(p) \equiv \psi_+(p) \times \psi_-(p),
  \]
  where $\psi_+(p)$ and $\psi_-(p)$ are defined as follows.
  \begin{align*}
    \psi_+(p) & \defeq  \sum_{h_+ : \prod_{n:\NN}P(n+1)}  \bigl((h_+(0) = f_0(p))         \\
    & \hspace {8em} \times \prod_{n:\NN}h_+(n+1) = f_n(h_+(n))\bigr) \\
    \psi_-(p) & \defeq  \sum_{h_- : \prod_{n:\NN}P(-n-1)} \bigl((h_-(0) = f_{-1}^{-1}(p))  \\
    & \hspace {8em} \times \prod_{n:\NN}h_-(n+1) = f_{-n-2}^{-1}(h_-(n))\bigr)
  \end{align*}
  The type $\varphi^{-1}(p)$ is contractible, because $\psi_+(p)$ and $\psi_-(p)$ are contractible.
  They are contractible because each specifies a certain function, namely $h_+$ or $h_-$,
  specifies that this function has a certain value at $0$,
  and prescribes the value of the function at all successors.
  Such a specification is unique by the universal property (induction) of $\NN$.
\end{proof}

Let us spell out the inverse function produced in the proof.
It maps $p : P(0)$ to a pair whose first component
is the function that takes $z:\zet$ to $f^z(p):P(z)$, where
\begin{alignat*}2
  f^0(p) &\defeq p, && \\
  f^{n+1}(p) &\defeq f_n(f^n(p)),&\quad&\text{for $n:\NN$,} \\
  f^{-n-1}(p) &\defeq f_{-n-1}^{-1}(f^{-n}(p)),&\quad&\text{for $n:\NN$.}
\end{alignat*}

\section{Main results}
\label{sec:ZTorsors}

In this section, pairs $(X,f)$ will be of type $\sum_{X:\UU}(X\to X)$.
Moreover, nested pairs will be written as tuples.
With these notational simplifications, we rephrase some definitions from
\cref{sec:preview-Z-tors}.
The formalization of \cref{sec:TorZ-recursion} can be found in \cite[\href{https://github.com/UniMath/UniMath/blob/master/UniMath/SyntheticHomotopyTheory/Circle.v}{Circle.v}]{UniMath},
and that of \cref{sec:TorZ-induction} can be found in in \cite[\href{https://github.com/UniMath/UniMath/blob/master/UniMath/SyntheticHomotopyTheory/Circle2.v}{Circle2.v}]{UniMath}.

\begin{definition}\label{def:TorZ}
  The pointed type of $\zet$-torsors is defined by
  \begin{align*}
    \TorZ &\defeq \sum_{(X,f)}\Trunc{(\zet,s)=(X,f)},  \\
    \pt & \defeq (\zet,s,\trunc{\refl{(\zet,s)}}) : \TorZ .
  \end{align*}
  The variables $X,Y,Z$ will be used for elements of $\TorZ$,
  as well as, by an abuse of notation, for their the first components.
  The equivalence $\ev_0$ is defined by
  \begin{align*}
    \ev_0: (\pt =_\TorZ \pt) & \to\zet \\
    p & \mapsto \pr_1(p)_*(0).
  \end{align*}
  The loop of $\TorZ$ is defined as ${\Zloop}\defeq\ev_0^{-1}(1)$,
  satisfying $\pr_1(\Zloop)_* = s$.
\end{definition}

The type $\TorZ$ is equivalent to the more traditionally defined
type of $\ZZ$-torsors, but is more parsimonious.\footnote{Our formalization, however, uses the traditional definition.}
(Traditionally, a $\ZZ$-torsor is defined as a nonempty set upon which the group $\ZZ$ acts freely and transitively.)
Another way to think of it is as the type of Cayley diagrams for $\ZZ$
with respect to the generator $1$.

We remark that the pointed type $\TorZ$ is connected, that is, $\Trunc{\pt=Z}$ for all $Z:\TorZ$.

\begin{lemma}\label{lem:dep-elim-TorZ}
If $P(Z)$ is a proposition for all $Z:\TorZ$, then $P(\pt)$ implies that $P(Z)$ holds for all $Z:\TorZ$.
\end{lemma}

\begin{proof}
Let $P(Z)$ be a proposition for all $Z:\TorZ$, and assume we have a proof
$p:P(\pt)$. Let $(X,f,t):\TorZ$, then $t :\Trunc{(\zet,s)=(X,f)}$.
Since $P(X,f,t)$ is a proposition, it suffices by \cref{lem:proptruncind}
to prove $P(X,f,\trunc{e})$ for all $e:(\zet,s)=(X,f)$.
By induction on $e$ we reduce the task to proving $P(\zet,s,\trunc{\refl{(\zet,s)}})$, which is
the same as $P(\pt)$, so $p$ provides the proof.
\end{proof}

The proof $q: \prod_{Z:\TorZ} P(Z)$ constructed above has the property that $q(\pt)\jdeq p$ if the computation rule for the induction principle
for propositional truncation holds.

Recall that the aim of this section is to show that $\TorZ$ satisfies the induction principle \cref{eq:circle-induction} for the circle.
The recursion principle is the non-dependent version of the induction principle,
namely that there is a function of the following type:
\begin{equation}
  \label{eq:circle-recursion}
  \thinmuskip=10mu              
  \prod_{A: \UU}
  \prod_{a: A}
  \prod_{p: a=_A a}
  \sum_{f: \TorZ \to A}
  \sum_{r: f(\pt)= a}
  \ap{f}(\Zloop) =^{\tilde A}_r p.
\end{equation}

Although the same method works to derive both the recursion and the induction principles,
we opt to do the recursion principle first, as it is slightly simpler,
and prepares the way for the more complicated induction principle.

\subsection{Recursion in \texorpdfstring{$\TorZ$}{TZ}}
\label{sec:TorZ-recursion}

Fix $A:\UU$, $a:A$ and $p: a=_A a$.
We want to construct a function $f$ from $\TorZ$ to $A$
that maps $\pt$ to $a$, as witnessed by some $r : f(\pt) = a$,
such that $\ap{f}(\Zloop) =^{\tilde A}_r p$.
As mentioned in \cref{sec:TorZ-circle},
the last requirement is equivalent to $\ap{f}(\Zloop) = r \ct (p \ct r^{-1})$.
This is because $\trp{\tilde A,r^{-1}}(p) = r \ct (p \ct r^{-1})$,
which in turns follows from the groupoid law $p = \refl{} \ct\, p$
upon induction on $r$.

All input data is present in $p$ and its type.
When defining types and functions depending on the input data,
we use $p$ in various denotations to express this dependence.

To be able to apply \cref{lem:dep-elim-TorZ}, we need to find a suitable proposition.
The idea is to find a correspondence\footnote{%
  A \emph{correspondence} (or \emph{span})
  from a type $T$ to a type $T'$ is a type $C$
  with projections $\pr_1 : C \to T$ and $\pr_2 : C \to T'$.}
from $\TorZ$ to $A$ whose first projection is an equivalence, thereby yielding a map from $\TorZ$ to $A$, cf.~\cref{fig:TorZ-recursion}.

\begin{figure}[t]
  \begin{tikzpicture}[commutative diagrams/every diagram,
    execute at begin node=$, execute at end node=$]
  \node (C) {%
    \sum\limits_{(X,f,t) : \TorZ} ~
    \sum\limits_{a':A} ~
    \sum\limits_{h: X\to a=a'} ~
    \prod\limits_{x:X} ~ h(f(x))=p\ct h(x)};
  \node (T2) [right=3em of C] {%
    \vphantom{\sum\limits_{a':A} ~
      \sum\limits_{h: X\to a=a'} ~
      \prod\limits_{x:X} ~ h(f(x))=p\ct h(x)}
    A};
  \node (T1) [below=of C] {\TorZ};

  \path[commutative diagrams/.cd, every arrow, every label]
    (C) edge[->>,"\pr_1","\equiv"'] (T1)
    (C) edge[->>,transform canvas={yshift=1ex},"\pr_2"] (T2)
    (T1) edge[dotted,bend right=15,"c_p"'] (T2)
    (T1) edge[dotted,to path={.. controls +(-5ex,0)
      and +(-8ex,-5ex) .. (\tikztotarget)
      \tikztonodes},"\vec c_p"] (C);
  \end{tikzpicture}

\caption{\label{fig:TorZ-recursion}Mapping torsors to $A$}
\end{figure}

\begin{definition}\label{def:guided-null-hmtps}
For every $(X,f)$, define
\begin{align*}\label{eq:TBD}
Q_p(X,f)&\defeq \sum_{a':A}~\sum_{h:X\to a=a'}~\prod_{x:X} h(f(x))=p\ct h(x).
\end{align*}
\end{definition}

\begin{lemma}\label{lem:guided-null-hmtps}
The type $Q_p(X,f)$ is contractible for all $(X,f,t):\TorZ$.
\end{lemma}
\begin{proof}
  By \cref{lem:dep-elim-TorZ} it suffices to prove that $Q_p(\zet,s)$ is contractible.
  Note that $Q_p(\zet,s)$ is the total space of the family $R_p : A \to \UU$ defined by
  \[
    R_p(a') \defeq \sum_{h:\zet \to a=a'}~\prod_{z:\zet} h(z+1)=p\ct h(z).
  \]
  Note furthermore that $\sum_{a':A} a=a'$ is contractible with center $(a,\refl{a})$.
  Thus, to show that $Q_p(\zet,s)$ is contractible,
  it suffices to define an equivalence
  \[
    \varphi_{a'} : \biggl(\sum_{h:\zet \to a=a'}\prod_{z:\zet} h(z+1)=p\ct h(z)\biggr)
    \xrightarrow{\sim}{} (a=a')
  \]
  for each $a':A$.
  The intention is now to invoke \cref{thm:integers-univ-symm}.
  Indeed, let us define the constant type family $P_{a'}(z)\defeq(a=a')$
  over $\zet$. Also, define $f_{a'}(z) : P_{a'}(z) \to P_{a'}(z+1)$
  by $f_{a'}(z)(q) \defeq p \ct q$ for all $z:\zet$ and $q: a=a'$.
  Then each $f_{a'}(z)$ is an equivalence (with inverse $q \mapsto p^{-1}\ct q$).
  Thus, applying \cref{thm:integers-univ-symm}
  shows that $\varphi_{a'}$ is an equivalence,
  where $\varphi_{a'}(h,q) \defeq h(0)$.
\end{proof}

A relevant observation at this point is that $Q_p(X,f)$ does not depend on
$t:\Trunc{(\zet,s)=(X,f)}$.
This means that we actually apply in the proof above the non-dependent version
of \cref{lem:dep-elim-TorZ}, for which the computation rule holds also in \UniMath{}.
For $Z \jdeq (X,f,t) : \TorZ$, let $\vec{c}_p(Z)$ denote the center of contractibility
of $Q_p(X,f)$ as constructed in the proof above.
  We introduce the notation $(c_p(Z),\tilde{c}_p(Z),\hat{c}_p(Z))\defeq\vec{c}_p(Z)$ for its components.
  The value of $\vec{c}_p(\pt)$ can be uncovered by a careful analysis of the steps of the proof.
First, the center of $\sum_{a':A} a=a'$ is $(a,\refl{a})$.
This center is pulled back by $\varphi_{a}$ to a center
$(a,c)$ of $Q_p(\zet,s)$, where $c$ is the center of
$\varphi_{a}^{-1}(\refl{a})$ coming from the proof
that $\varphi_{a}$ is an equivalence. The latter proof
is the above instance of \cref{thm:integers-univ-symm}.
Unraveling this instance, and using the remark at the
end of the proof of \cref{thm:integers-univ-symm},
tells us that $c$ is a pair $(h,q)$ with $h(z)\jdeq p^z$
for all $z:\zet$. Indeed, $\varphi_{a}(h,q) = h(0) \jdeq \refl{a}$.
Moreover, $q$ has type $\prod_{z:\zet} h(z+1)=p\ct h(z)$.
Wrapping up, $\vec c_p(\pt) = (a,h,q)$,
with judgmental equality if~\cref{cond:JNE} holds.

The analysis in the previous paragraph
means we have achieved one of our goals,
namely that the function $c_p$ from $\TorZ$ to $A$
maps $\pt$ to $a$, definitionally if~\cref{cond:JNE} holds.
In any case, let $r \defeq \tilde c_p(\pt,0)^{-1} : c_p(\pt) = a$,
which reduces under~\cref{cond:JNE} to $\refl{a}$.
We will now deal with the other goal,
namely that $c_p$ acting on $\Zloop$ yields $r \ct (p \ct r^{-1})$.

\begin{lemma}\label{lem:ap-c-tilde-c}
For all $X,Y:\TorZ$, $e: X=Y$ and $x:X$ we have
$\ap{c_p}(e) = \tilde c_p(X,x)^{-1}\ct \tilde c_p(Y,\ptoe{e}(x))$,
where $\ptoe{e}\defeq \pr_1(e)_* : X\to Y$.
\end{lemma}
\begin{proof}
By using induction on $e$ we only have to check the case where
$X\jdeq Y$ and $e\jdeq\refl{X}$. In this case $\ap{c_p}(e)$ is
$\refl{c_p(X)}$. On the right-hand side we get $\ptoe{e}(x)\jdeq x$,
and hence this side simplifies to a reflexivity path of
the correct type, as $\tilde c_p(X,x)$ has type $a=c_p(X)$.
\end{proof}

We apply the above lemma with $X\jdeq Y\jdeq \pt$ and $e\jdeq{\Zloop}: \pt=\pt$.
Then we have $\ptoe{e}(z)=s(z)=z+1$.
Note that $\hat c_p(\pt,0) : \tilde c_p(\pt,1) = p \ct \tilde c_p(\pt,0)$.
Hence, taking $z\defeq 0$ in \cref{lem:ap-c-tilde-c}, it follows that
\[
  \ap{c_p}(\Zloop) = \tilde c_p(\pt,0)^{-1} \ct \tilde c_p(\pt,1)
  = r \ct (p \ct r^{-1}).
\]
This means we have achieved our second goal as well,
and we've produced an element of the type \eqref{eq:circle-recursion}.

\subsection{Induction in \texorpdfstring{$\TorZ$}{TZ}}
\label{sec:TorZ-induction}

Fix $A:\TorZ\to\UU$, $a:A(\pt)$, and $p: a=^A_{\Zloop} a$.
On the basis of this input data, we will construct a function $f$ of
type $\prod_{Z:\TorZ} A(Z)$ that maps $\pt$ to $a$,
as witnessed by some $r : f(\pt) = a$,
such that $\apd{f}(\Zloop) =^{\tilde A}_r p$.
We follow the pattern of the non-dependent case
in \cref{sec:TorZ-recursion}, but keep in mind that
$A$ is now not constant and $p$ is a \emph{path over a path}.
{We make extensive use of the functions and lemmas from \cref{sec:pathovers}.}

The following lemma follows from the fact that $\zet$ is a set
and $s: \zet\to\zet$ is an equivalence.

\begin{lemma}\label{lem:paths-in-TorZ}
  Suppose $q : (\zet,s)=(X,f)$.  Then $X$ is a set and $f: X\to X$ is an equivalence.
  Morever, with $\ptoe{q}\defeq \pr_1(q)_*$ the equivalence induced by $q$, we have
  $f^n(x) = (\ptoe{q}\circ s^n \circ \ptoe{q}^{-1})(x) = \ptoe{q}(\ptoe{q}^{-1}(x)+n)$,
  for all $n:\zet$ and $x:X$.
\end{lemma}

Note that for fixed $x:X$ the expression $\ptoe{q}^{-1}(x)+n$ can be seen as
the function shifting $n:\zet$ by $\ptoe{q}^{-1}(x)$ positions, indeed an equivalence.
Hence $f^n(x)$ as a function of $n$ is an equivalence from $\zet$ to $X$.
Recall that we may denote $f^n(x)$ by $n+x$,
as $X$ is a $\ZZ$-torsor via $f$.

\begin{definition}\label{def:loop-s-iterated}
  For every $Z\defeq(X,f,t):\TorZ$ and $x:X$,
  define $s^Z_x: \pt =_\TorZ Z$ by the equivalence
  $e_x(n)\defeq f^n(x)$ using the univalence axiom. Indeed,
  $f\circ e_x = e_x \circ s$, as both functions map $n$ to $f^{n+1}(x)$.
\end{definition}

Applying \cref{thm:integers-univ-symm}, we will need two auxiliary
results about the paths $s^\pt_x$, one for $x=0$ and
the other for the (symmetric) induction step.
In \cref{def:loop-s-iterated}, if $Z\jdeq\pt$ and $x=0$, we get $e_0 = \id$.
Applying the univalence axiom gives thus the first result.

\begin{lemma}\label{lem:s-pt-zero}
  There is a path $\gamma_0 : \refl{\pt} = s^\pt_0$.
\end{lemma}

For the second result, note that prefixing $s^Z_x$ by ${\Zloop}$
amounts to precomposing the equivalence $e_x$ with $s$.
We have $(e_x\circ s)(n) = e_x(n+1) = f^{n+1}(x) = f^n(f(x)) =e_{f(x)}(n)$,
so $e_x\circ s = e_{x+1}$. Applying the univalence axiom we get:

\begin{lemma}\label{lem:loop-s-iterated}
  For every $Z\defeq(X,f,t):\TorZ$ and $x:X$, we have a path
  ${\delta^Z_x}: {\Zloop} \ct s^Z_x = s^Z_{x+1}$.
\end{lemma}

Now we are ready to derive the induction principle using the same
technique as for the recursion principle. We reuse notations as much as
possible, but take care that all types are different.

\begin{definition}\label{def:guided-null-hmtps-dep}
For every $Z\defeq(X,f,t):\TorZ$, define
\begin{align*}\label{eq:TBD}
Q_p(Z)&\defeq \sum_{a' : A(Z)}~\sum_{h : \prod_{x:X} a =^A_{s^Z_x} a'}~\prod_{x:X} h(f(x)) = \cp{\delta^Z_x}(p\cto h(x)),
\end{align*}
where $\delta^Z_x$ comes from \cref{lem:loop-s-iterated}.
\end{definition}
Note that, unlike the $Q_p$ from~\cref{def:guided-null-hmtps},
this version depends crucially on the $t$-component of $Z$ through
both $s^Z$ and $\delta^Z$.

\begin{lemma}\label{lem:guided-null-hmtps-dep}
  For every $Z:\TorZ$, the type $Q_p(Z)$ is contractible.
\end{lemma}
\begin{proof}
  By \cref{lem:dep-elim-TorZ} it suffices to prove that $Q_p(\pt)$ is contractible.
  We have $Q_p(\pt)\jdeq \sum_{a':A(\pt)} R(a')$ for $R : A(\pt) \to \UU$ defined by
  \[
    R_p(a') \defeq \sum_{h : \prod_{z:\zet} a =^A_{s^\pt_z} a'}~
    \prod_{z:\zet} h(z+1) = \cp{\delta^\pt_z}(p \cto h(z)).
  \]
  We show that $\sum_{a':A(\pt)} (a =^A_{s^\pt_0} {a'})$ is contractible.
  Let $\reflo{a}$ be the reflexivity path at $a$ over $\refl{\pt}$.
  Note that $\sum_{a':A(\pt)} (a =^A_{\refl{\pt}} a') \jdeq \sum_{a':A(\pt)} (a = a')$
  is contractible with center $(a,\reflo{a})$. By \cref{lem:s-pt-zero}
  $\sum_{a':A(\pt)} (a =^A_{s^\pt_0} {a'})$ is contractible with center
  $(a,\cp{\gamma_0}(\reflo{a}))$.

  Thus, to show that $Q_p(\pt)$ is contractible,
  it suffices to define an equivalence
  \[
    \varphi_{a'} : R_p(a') \xrightarrow{\sim}{} (a =^A_{s^\pt_0} {a'})
  \]
  for each $a' : A(\pt)$.
  We again invoke \cref{thm:integers-univ-symm},
  this time with the family $P_{a'} : \zet \to \UU$ given
  by $P_{a'}(z) \defeq (a =^A_{s^\pt_z} {a'})$
  and the equivalences $f_{a'} : \prod_{z:\zet} P_{a'}(z)\equiv P_{a'}(z+1)$
  given by $f_{a'}(z) \defeq \cp{\delta^\pt_z}(p \cto (\blank))$.
  Thus, applying \cref{thm:integers-univ-symm}
  shows that $\varphi_{a'}$ is an equivalence,
  where $\varphi_{a'}(h,q) \defeq h(0)$.
\end{proof}

To simplify notations, we will now use variables $X,Y:\TorZ$,
and also write $X,Y:\UU$ for the underlying types.

Let $\vec c_p(X)$ denote the center of contraction of $Q_p(X)$ for $X:\TorZ$.
Again, we write $(c_p(X),\tilde{c}_p(X),\hat{c}_p(X))\defeq\vec{c}_p(X)$
for the components, where
\begin{align*}
         c_p &: \prod_{X:\TorZ} A(X),\\
  \tilde c_p &: \prod_{X:\TorZ}\prod_{x:X} a =^A_{s^X_x} c_p(X), \\
    \hat c_p &: \prod_{X:\TorZ}\prod_{x:X} \tilde c_p(X,1+x)
               = \cp{\delta^X_x}(p \cto \tilde c_p(X,x)).
\end{align*}
In particular, $q\defeq \tilde c_p(\pt,0) : a =^A_{s^\pt_0} c_p(\pt)$,
so we can define $r\defeq \cp{(\gamma_0^{-1})^-}(q\invo) : c_p(\pt) = a$
(recall~\cref{lem:inv2-change-path}).

We now proceed to establish that $\apd{c_p}(\Zloop) =^{\tilde A}_r p$.
Again we work be elaborating $\apd{c_p}(\Zloop)$.

\begin{lemma}\label{lem:s-X-x-*-e}
  For all $X,Y:\TorZ$, $e: X=Y$ and $x:X$ we have a path
  $\varepsilon_{e,x} : (s^X_x)^{-1} \ct s^Y_{\ptoe{e}(x)} = e$,
  where $\ptoe{e}\defeq \pr_1(e)_* : X\to Y$.
\end{lemma}
\begin{proof}
By induction on $e$ it suffices to give
$\varepsilon_{\refl{X},x} : (s^X_x)^{-1} \ct s^X_x = \refl{X}$,
since $\pr_1(\refl{X})_*(x) \jdeq x$.
Hence we set $\varepsilon_{\refl{X},x} \defeq \iota(s^X_x)$,
with $\iota$ as in \cref{lem:invlaw-change-path}.
\end{proof}

\begin{lemma}\label{lem:apd-c-tilde-c}
For all $X,Y:\TorZ$, $e: X=Y$ and $x:X$ we have
$\apd{c_p}(e) = \cp{\varepsilon_{e,x}}
(\tilde c_p(X,x)\invo\cto \tilde c_p(Y,\ptoe e(x)))$,
where $\ptoe e$ is as above.
\end{lemma}
\begin{proof}
Induction on $e$ reduces the proof to the case $e\jdeq\refl{X}: X=X$,
which follows from \cref{lem:invlaw-change-path} and \cref{lem:s-X-x-*-e}.
It remains to check that the general statement of the lemma is well-typed.
We have the following paths over paths.
\begin{align*}
\apd{c_p}(e) &: c_p(X) =^A_e c_p(Y)\\
\tilde c_p(Y,\ptoe e(x)) &:  a =^A_{s^Y_{\ptoe e(x)}} c_p(Y)\\
\tilde c_p(X,x)\invo &: c_p(X) =^A_{(s^X_x)^{-1}} a\\
\tilde c_p(X,x)\invo \cto c_p(Y,\ptoe e(x)) &:
     c_p(X) =^A_{(s^X_x)^{-1}\ct s^Y_{\ptoe e(x)}} c_p(Y)
\end{align*}
In order to make ends meet between the first and the fourth typing
we invoke \cref{lem:s-X-x-*-e} and \cref{def:pathover-change-path}.
\end{proof}

We're now ready to calculate $\apd{c_p}(\Zloop)$.
A great simplification is obtained by using that $\TorZ$ is a groupoid:
all $2$-paths having the same endpoints are equal.
Hence the functions $\cp{\alpha}$ only depend on the path-type
of $\alpha$.
Also, as $\cp{\refl{}}$ is the identity, so is any $\cp{\alpha}$
when the endpoints of $\alpha$ are definitionally equal.
We make extensive use of this simplification.

Abbreviate $s_0 \defeq s^\pt_0$, $s_1 \defeq s^\pt_1$,
then $\varepsilon_{\Zloop,0} : s_0^{-1} \ct s_1 = {\Zloop}$.
We split the latter identity in a sequence of identities:
\[
s_0^{-1} \ct s_1 \stackrel{\alpha}{=}
\refl{\pt}\ct(\Zloop\ct s_0) \stackrel{\beta}{=}
\refl{\pt}\ct(\Zloop\ct\refl{\pt})\stackrel{\gamma}{=} {\Zloop}
\]
Here $\alpha\defeq\apc((\gamma_0^{-1})^-,\delta_0^{-1})$ and
$\beta\defeq\apc(\refl{\refl{\pt}},\apc(\refl{\Zloop},\gamma_0^{-1}))$ are constructed with $\apc$ so as to
enable the application of \cref{lem:compo-change-path}.
Also, for $\gamma$ we may take $\gamma(\Zloop)$
from~\cref{lem:unitlaw-change-path}.
We now calculate, recalling that $r\jdeq\cp{(\gamma_0^{-1})^-}(q\invo)$
and $q\jdeq\tilde c_p(\pt,0)$:
\allowdisplaybreaks
\begin{align*}
  \apd{c_p}(\Zloop)
  &= \cp{\varepsilon_{\Zloop,0}}
    \bigl(\tilde c_p(\pt,0)\invo \cto \tilde c_p(\pt,1)\bigr) \\
  &= \cp{\alpha\ct\beta\ct\gamma}
    \bigl(q\invo \cto \cp{\delta_0}(p \cto q)\bigr) \\
  &= \cp{\gamma}\biggl(\cp{\beta}\Bigl(\cp{\alpha}
    \bigl(q\invo \cto \cp{\delta_0}(p \cto q)\bigr)\Bigr)\biggr) \\
  &= \cp{\gamma}\biggl(\cp{\beta}\Bigl(
    r \cto \cp{\delta_0^{-1}}\bigl(\cp{\delta_0}(p \cto q)\bigr)
    \Bigr)\biggr) \\
  &= \cp{\gamma}\Bigl(\cp{\beta}\bigl(
    r \cto (p \cto q)\bigr)\Bigr) \\
  &= \cp{\gamma}\Bigl(r \cto \bigl(p \cto \cp{\gamma_0^{-1}}(q)\bigl)\Bigr)
   = \cp{\gamma}\Bigl(r \cto \bigl(p \cto r^{-1}\bigl)\Bigr)
\end{align*}
To conclude that $\apd{c_p}(\Zloop) =^{\tilde A}_r p$, we only need
a final auxiliary lemma that describes what happens when we
transport $p$ backwards along $r$ in the family $\tilde A$.
\begin{lemma}
  For any $X : \UU$, $B : X \to \UU$, $x:X$, $b,c:B(x)$,
  $s : x=x$, $r : b =^B_{\refl{x}} c$, and $q : c =^B_s c$
  we have
  \[
    \trp{\tilde B,r^{-1}}(q) =
    \cp{\gamma(s)}(r \cto (q \cto r^{-1})),
  \]
  where $\tilde B(y) \defeq (y =^B_s y)$ for $y:B(x)$,
  and $\gamma(s)$ is as in~\cref{lem:unitlaw-change-path}.
\end{lemma}
The proof is by induction on $r$, followed by an appeal
to~\cref{lem:unitlaw-change-path}.

If~\cref{cond:JDE}, and not just~\cref{cond:JNE}, holds,
then we see that $c_p(\pt)\jdeq a$, and $r=\refl{a}$,
since $\vec c_P(\pt)$ reduces to
$\varphi_{a}^{-1}(a,\cp{\gamma_0}(\reflo{a}))$,
which is a triple $(a,h,q)$, where $h(0)=\cp{\gamma_0}(\reflo{a})$.

\section{Interpretation in higher toposes}
\label{sec:topos}

Voevodsky's pioneering work~\cite{1211.2851} constructed interpretations of the rules of univalent foundations (but not the entire formal system~\cite{voevodsky:not-interp}) in the Quillen model category of simplicial sets, which is a presentation of the fundamental $(\infty,1)$-topos of $\infty$-groupoids.
After a decade of further work, this interpretation has now been extended to include a class of model categories presenting all $(\infty,1)$-toposes by Shulman~\cite{shulman:univinj}, and made into an interpretation of the entire formal system by Brunerie, de Boer, Lumsdaine, and M\"{o}rtberg~\cite{initiality}.
Thus, we can now say conclusively that the result of our paper yields a theorem about all $(\infty,1)$-toposes.

However, this theorem requires a bit of unpacking to make it look familiar to higher topos theorists.
In particular, since all $(\infty,1)$-toposes are cocomplete, the more usual way to define an internal ``circle object'' in such a topos would be as a (homotopy) colimit: specifically, the coequalizer of two copies of the identity map of the terminal object (corresponding to the presentation of a circle as a cell complex with one 0-cell and one 1-cell).
The natural theorem to expect would then be that \emph{this} circle object is a classifier for $\ZZ$-torsors.

One way to obtain such a result from our theorem would be to observe that according to the interpretation of higher inductive types in higher toposes constructed by~\cite{1705.07088}, the higher inductive $\Sc$ is a presentation of the above homotopy colimit.
Since we have shown that our $\TorZ$ has the same induction principle that $\Sc$ has by definition, they must be equivalent.
Thus, since our $\TorZ$ classifies $\ZZ$-torsors by construction, so does $\Sc$ and hence so does the circle object.

However, a more direct approach is also possible, which avoids discussing higher inductive types at all:\footnote{As of this writing, higher inductive types are not yet included in the work of~\cite{initiality}.  Also the models of universes in~\cite{shulman:univinj} are not yet known to be closed under parametrized higher inductive types, although this is not a problem in our situation since $\Sc$ is not parametrized.} we can use our theorem to show that the interpretation of $\TorZ$ in an $(\infty,1)$-topos has the universal property of the homotopy colimit that defines a circle object.
It will then follow that since it classifies $\ZZ$-torsors, so does any other such homotopy colimit.

\begin{theorem}\label{thm:torz-coeq}
  In any $(\infty,1)$-topos $\sE$, there is a coequalizer diagram $1 \rightrightarrows 1 \to \TorZ$.
\end{theorem}
\begin{proof}
  Half of the proof takes place inside of type theory and the other half in a model category.
  However, the first half has already been done by Sojakova~\cite{sojakova:hits-hias},
  as mentioned in \cref{sec:circles}.
  Her main theorem~\cite[Theorem 50]{sojakova:hits-hias} then implies that if a circle algebra $\mathcal C\jdeq(C,c,s)$ satisfies the induction principle,
  then it is \emph{homotopy-initial} in that for any other circle algebra $\mathcal A(A,a,p)$
  the type of circle algebra homomorphisms~\eqref{eq:circle-hom}
  is contractible, i.e., there is an element of the type
  \[
    \prod_{\mathcal A:\CircleAlg} \iscontr\bigl(\Shom(\mathcal C,\mathcal A)\bigr).
  \]
It follows that our $\TorZ$ is homotopy-initial in this sense.

For the second half of the proof, suppose we have a Quillen model category $\cE$ that presents our $(\infty,1)$-topos $\sE$.
We must show that for any object $A$ of $\sE$, the diagram of hom-spaces ($\infty$-groupoids)
\[ \sE(\TorZ,A) \to \sE(1,A) \rightrightarrows \sE(1,A) \]
is a homotopy equalizer, i.e., that the map from $\sE(\TorZ,A)$ to the homotopy equalizer of two copies of the identity map of $\sE(1,A)$ is an equivalence.
It suffices to show that the homotopy fiber of this map over any point is contractible, which is to say that given any point $a:1\to A$ and homotopy $p:a\sim a$, the space of maps $f:\TorZ\to A$ equipped with a homotopy $r : h(\pt) \sim a$ and a higher homotopy $h(\Zloop) \ct r \sim r \ct p$ is contractible.

Now, homotopies in the $(\infty,1)$-category $\sE$ can always be presented by \emph{right homotopies} in the model category $\cE$, meaning maps into a path object.
Thus, our object $A$ of $\sE$ with $a$ and $p$ can be presented by an object of $\cE$, which we also denote $A$, with a point $a:1\to A$ in $\cE$ and a right homotopy $p:1\to P A_{(a,a)}$, where $P A$ is the path object of $A$ and $P A_{(a,a)}$ is its pullback along $(a,a):1\to A\times A$.
Since path objects supply the interpretation of identity types, this corresponds to a type $A$ with an element $a:A$ and path $p:a=_A a$, i.e., a circle algebra.
By a similar argument, the homotopy fiber that we want to prove contractible is equivalent to the hom-space $\sE(1, H_{\TorZ,A})$.
Thus, it suffices to observe that if $B$ is an object with a point $1\to \iscontr(B)$, then $B$ is equivalent to the terminal object (a detailed proof can be found in~\cite[Lemma 4.1]{shulman:elreedy}).
\end{proof}

We should also explain more carefully why our $\TorZ$ ``classifies $\ZZ$-torsors'' in the $(\infty,1)$-categorical sense.
Importantly, the relevant notion of ``$\ZZ$-torsor'' is the ``local'' topos- and sheaf-theoretic one: an object $X$ of a slice $(\infty,1)$-topos $\sE/A$ is a \emph{$\ZZ$-torsor} if and only there is an effective epimorphism\footnote{An \emph{effective epimorphism} is a morphism that is the quotient of its kernel.  In an $(\infty,1)$-topos the relevant ``kernels'' are simplicial objects; see~\cite[Section 6.2.3]{lurie:higher-topoi}.} $p:B \twoheadrightarrow A$ such that $p^*X$ is isomorphic over $B$ to $\ZZ\times B$.

\begin{theorem}\label{thm:torz-classif}
  For any object $A$ of an $(\infty,1)$-topos $\sE$, the hom-space $\sE(A,\TorZ)$ is naturally equivalent to the $\infty$-groupoid of $\ZZ$-torsors in $\sE/A$.
\end{theorem}
\begin{proof}
  Recall that $\TorZ$ is defined as $\sum_{X:\UU}\sum_{f:X\to X} \Trunc{(\zet,s)=(X,f)}$, so that it comes with a sequence of projections
  \[ \TorZ \longrightarrow \Big(\sum_{X:\UU} (X\to X)\Big) \longrightarrow \UU .\]
  We will start by characterizing what $\UU$ classifies, then $\sum_{X:\UU} (X\to X)$, then finally $\TorZ$.

For the first, the univalence axiom implies that $\UU$ is an object classifier, in the sense that $\sE(A,\UU)$ is naturally equivalent, by pullback of the canonical map $\UUp\to \UU$, to a full sub-$\infty$-groupoid of the core of $\sE/A$ (the fiberwise ``small'' objects).
This means that for any $f,g:A\to \UU$, the induced map from the space of homotopies $f\sim g$ to the space of equivalences $f^*\UUp \equiv g^*\UUp$ in $\sE/A$ is an equivalence.
This holds because the former is the space of lifts of $(f,g):A\to \UU\times \UU$ to the path-object $P\UU$, while the latter is (e.g., by~\cite[Lemma 4.3]{shulman:elreedy}) the space of sections of $\mathrm{Equiv}_A(f^*\UUp,g^*\UUp)$, or equivalently the lifts of $(f,g)$ to $\mathrm{Equiv}_{\UU\times \UU}(\pi_1^*\UUp,\pi_2^*\UUp)$, and the univalence axiom says precisely that $P\UU$ is equivalent to $\mathrm{Equiv}_{\UU\times \UU}(\pi_1^*\UUp,\pi_2^*\UUp)$ over $\UU\times\UU$.

Secondly, $\sum_{X:\UU} (X\to X)$ is the exponential in $\sE/\UU$ of $\UUp$ by itself.
Thus, by the pullback-stability and universal property of exponentials, the space of lifts of $f:A\to \UU$ to $\sum_{X:\UU} (X\to X)$ is equivalent to the space of endomorphisms of the corresponding object $f^*\UUp$ of $\sE/A$.
Hence $\sum_{X:\UU} (X\to X)$ classifies small objects equipped with an endomorphism, in that $\sE(A, \sum_{X:\UU} (X\to X))$ is naturally equivalent to the $\infty$-groupoid of small objects of $\sE/A$ with an endomorphism.

Thirdly, for $\TorZ$ we need to consider the topos-theoretic interpretation of propositional truncation.
We can ignore its particular construction in \UniMath{} and focus on its universal property, which says that it is a reflection into propositions, i.e., for any proposition $P$ we have $(\Trunc{X}\to P) \equiv (X\to P)$.
Interpreted fiberwise in an $(\infty,1)$-topos, this says that if we represent a morphism $X\to A$ as the first projection from a dependent sum, $(\sum_{a:A} X(a)) \to A$, then the corresponding projection $(\sum_{a:A} \Trunc{X(a)})\to A$ is its reflection into the sub-$(\infty,1)$-category of $\sE/A$ consisting of the monomorphisms, i.e., maps $P\to A$ whose diagonal $P\to P\times_A P$ is an equivalence.
By~\cite[Example 5.2.8.16 and Corollary 6.5.1.14]{lurie:higher-topoi} (in the case $n=-1$), the effective epimorphisms and monomorphisms in an $(\infty,1)$-topos form a factorization system, and hence in particular the map $X\to (\sum_{a:A} \Trunc{X(a)})$ is an effective epimorphism.

Now by the definition of $\TorZ$, we have an (effective epimorphism, monomorphism) factorization $W \twoheadrightarrow \TorZ \rightarrowtail \sum_{X:\UU} (X\to X)$, where $W = \sum_{(X,f)} (\zet,s)=(X,f)$ is (by arguments like those above) a classifier for small objects $X$ equipped with an endomorphism $f$ and a \emph{specified} equivalence to $\zet$ respecting the endomorphisms.
In particular, when the universal object-with-endomorphism over $\sum_{X:\UU} (X\to X)$ is pulled back to $\TorZ$, then there exists an effective epimorphism onto $\TorZ$ (namely $W\to \TorZ$) such that when it is pulled back further along that morphism it becomes equivalent to $(\zet,s)$.
Thus, this universal object is a $\ZZ$-torsor, and hence so is any pullback of it.

Conversely, suppose $(X,f)$ is a $\ZZ$-torsor in $\sE/A$, so there is an effective epimorphism $p:B\twoheadrightarrow A$ such that $p^*(X,f)$ is equivalent to $(\zet,s)$.
Since $\zet$ is a small object, this implies that $X$ has small fibers, hence is classified by some map $A\to \UU$.
Moreover, the assumption implies that when the classifying map $A\to \sum_{X:\UU} (X\to X)$ is composed with $p$, it lifts to $W$; so we have the following diagram:
\[
\begin{tikzcd}
  B \ar[dd,->>] \ar[r] & W \ar[d,->>] \\
  & \TorZ \ar[d,>->] \\
  A \ar[r] & \sum_{X:\UU} (X\to X)
\end{tikzcd}
\]
Thus, since effective epimorphisms and monomorphisms form a factorization system, there is an essentially unique lift $A\to \TorZ$.
So we have shown that an object-with-endomorphism is a $\ZZ$-torsor precisely when its classifying map $A\to \sum_{X:\UU} (X\to X)$ lifts to $\TorZ$.
Since $\TorZ \to \sum_{X:\UU} (X\to X)$ is a monomorphism, this means that $\TorZ$ classifies $\ZZ$-torsors.
\end{proof}

Combining \cref{thm:torz-coeq,thm:torz-classif}, we see that any circle object in an $(\infty,1)$-topos is equivalent to $\TorZ$ and hence classifies $\ZZ$-torsors.

\section{Conclusion and future research}
\label{sec:conclusion}

We have proved, {for any type family $A\to\TorZ$,}
the induction principle of the circle for $\TorZ$:
\[
{\ind}_A : \sum_{a:A(\pt)}(a=^A_\Zloop a) ~~\to~~ \prod_{Z:\TorZ} A(Z),
\]
with ${\ind}_A$ mapping $(a,p)$ to $c_p$ satisfying $c_p(\pt) \jdeq a$
and $\apd{c_p}(\Zloop) = p$.

It would be interesting to see whether our method can be generalised
from $\TorZ$ to the type $BG$ of $G$-torsors, where $G$ is a free group with
a set of generators, $S$, with decidable equality.
Explicitly, we expect that it is possible in our setting to
prove that $BG$ satisfies the induction principle for a higher inductive
type with a point constructor $\pt : BG$ and a path constructor
$\Zloop_{\blank} : S \to (\pt =_{BG} \pt)$,
\[
  \prod_{A: BG\to\UU}~
  \prod_{a: A(\pt)}~
  \prod_{p: \prod_{s:S} a=^A_{\Zloop_s} a}~
  \sum_{f: \prod_{z:BG} A(z)}~
  \sum_{r: f(\pt)= a}~
  \prod_{s:S}\apd{f}(\Zloop_s) =^{\tilde A_s}_r p_s,
\]
where $\tilde A_s(y) \defeq (y =^A_{\Zloop_s} y)$ for $y : A(\pt)$.

\section*{Acknowledgements}

Bezem, Buchholtz, and Grayson acknowledge the support of the Centre for Advanced Study (CAS)
at the Norwegian Academy of Science and Letters
in Oslo, Norway, which funded and hosted the research project Homotopy Type Theory and Univalent Foundations during the academic year 2018/19.
Grayson also acknowledges the support of the Air Force Office of Scientific Research, through a grant to Carnegie Mellon University.
Shulman was supported by The United States Air Force Research
    Laboratory under agreement number FA9550-15-1-0053.  The
    U.S. Government is authorized to reproduce and distribute reprints
    for Governmental purposes notwithstanding any copyright notation
    thereon.  The views and conclusions contained herein are those of
    the author and should not be interpreted as necessarily
    representing the official policies or endorsements, either
    expressed or implied, of the United States Air Force Research
    Laboratory, the U.S. Government, or Carnegie Mellon University.

\raggedright
\printbibliography

\end{document}